\documentclass[a4paper,11pt]{amsart}
\usepackage[colorlinks,linkcolor=blue,citecolor=blue]{hyperref}
\usepackage{latexsym, amssymb, amsmath, amscd, amsthm, mathrsfs, bbm}
\usepackage[all, knot]{xy}
\xyoption{arc}

\usepackage{anysize}\marginsize{22mm}{22mm}{25mm}{25mm}
\allowdisplaybreaks[4]

\def \k {\mathbbm{k}}
\def \End {\operatorname{End}}

\def \dim {\operatorname{dim}}

\def \R {\mathbbm{R}}
\def \C {\mathbbm{C}}

\numberwithin{equation}{section}
\numberwithin{table}{section}
\numberwithin{equation}{section}
\newtheorem{theorem}{Theorem}[section]
\newtheorem{lemma}[theorem]{Lemma}
\newtheorem{proposition}[theorem]{Proposition}
\newtheorem{corollary}[theorem]{Corollary}
\newtheorem{definition}[theorem]{Definition}
\newtheorem{example}[theorem]{Example}

\newtheorem{remarks}[theorem]{Remarks}

\title{Centers of Multilinear Forms and Applications}
\thanks{Supported by NSFC 11911530172, 11971181 and 11971449.}



\author{Hua-Lin Huang, Huajun Lu, Yu Ye and Chi Zhang}

\address{Huang and Lu: School of Mathematical Sciences, Huaqiao University, Quanzhou 362021, China}
\email{hualin.huang@hqu.edu.cn, huajun@hqu.edu.cn}

\address{Ye: School of Mathematical Sciences, University of Science and Technology of China, Hefei 230026, China}
\email{yeyu@ustc.edu.cn}

\address{Zhang: Department of Mathematics, Northeastern University, Shenyang 110819, China}
\email{zhangchi@mail.neu.edu.cn}

\date{}                                           

\begin{document}

\begin{abstract}
In this paper we study the center algebras of multilinear forms. It is shown that the center of a nondegenerate multilinear form is a finite dimensional commutative algebra and can be effectively applied to its direct sum decompositions. As an application of the algebraic structure of centers, we also show that almost all multilinear forms are absolutely indecomposable. The theory of centers can be extended to multilinear maps and be applied to their symmetric equivalence. Moreover, with a help of the results of symmetric equivalence, we are able to provide a linear algebraic proof of a well known Torelli type result which says that two complex homogeneous polynomials with the same Jacobian ideal are linearly equivalent.
\end{abstract}

\maketitle

\section{Introduction}
Let $V$ be a vector space over a field $\k$ of dimension $n.$ Let $d \geq 3$ be a positive integer. A $d$-linear form on $V$ is a multilinear mapping $\Theta:V^d=V\times \cdots \times V \to \k$ and is denoted by $(V,\Theta)$ or $\Theta$ for short. Take a basis $e_1, e_2, \dots, e_n$ of $V$ and set $a_{i_1 i_2 \cdots i_d}=\Theta(e_{i_1}, e_{i_2}, \dots, e_{i_d}).$ The resulting tensor $A=(a_{i_1 i_2 \cdots i_d})_{1 \le i_1, i_2, \dots, i_d \le n}$ is called the associated tensor of $(V,\Theta)$ under the basis $e_1, e_2, \dots, e_n.$ A fundamental problem in invariant theory and multilinear algebra is finding canonical forms for multilinear forms under base change, or equivalently, canonical forms of tensors under congruence by invertible matrices.

Unlike bilinear forms, it seems hopeless to find complete representatives for $d$-linear forms, see for example \cite{bs0, fgs}. One of our main concerns is direct sum decompositions of multilinear forms, that is to find whether there exist nonzero subspaces $V_1, V_2, \dots, V_m$ of $(V, \Theta)$ such that $V=V_1 \oplus V_2 \oplus \cdots \oplus V_m$ and $\Theta(v_1, \cdots, v_d)=0$ unless all the $v_i$'s are in the same $V_j$ for some $j.$ In terms of tensors, this is equivalent to finding an invertible matrix $P$ such that the congruent tensor $AP^d$ is block diagonal. This is a natural problem as direct sum decompositions provide dimension reduction for multilinear forms.

In \cite{hlyz,hlyz2}, we investigated direct sum decompositions of symmetric multilinear forms via Harrison's theory of centers \cite{h1}. The authors showed that the problem can be boiled down to some standard tasks of linear algebra, specifically the computations of eigenvalues and eigenvectors. The main aim of the present paper is to extend \cite{hlyz,hlyz2} to the situation of multilinear forms.

We generalize the key notion of centers as follows.
\begin{definition} \label{cc} 
Given a $d$-linear form $(V, \Theta),$ set
\begin{equation}
Z(V,\Theta) \colon = \left\{ \phi \in \operatorname{End}(V) \left|
\begin{array}{c}
  \Theta(v_1,  \cdots, \phi(v_{i}), \cdots, v_{j}, \cdots v_d) \\
  =\Theta(v_1, \cdots, v_i,\cdots, \phi(v_j),\cdots, v_d), \\
  \forall 1\leq i,j\leq d, \ \ \forall v_1, \cdots, v_d \in V
\end{array}
\right. \right\} \end{equation}
and call it the center of $(V, \Theta).$
\end{definition}
\noindent Elements of centers for multilinear forms were also defined and applied to direct sum decomposition in \cite{bs}, where they were called self-adjoint linear mappings. However, the algebraic structure of all centers, or self-adjoint elements were not considered therein.

The centers of multilinear forms enjoy the same properties as those of symmetric multilinear forms or equivalently homogeneous polynomials, cf. \cite{h1, hlyz, hlyz2}.
\begin{theorem}\label{cd}
Suppose $(V,\Theta)$ is a nondegenerate $d$-linear form. Then
\begin{itemize}
\item[(1)] The center $Z(V,\Theta)$ is a commutative algebra.
\item[(2)] There is a one-to-one correspondence between  direct sum decompositions of $(V,\Theta)$ and complete sets of orthogonal idempotents of $Z(V,\Theta).$
\item[(3)] The decomposition of $(V,\Theta)$ into a direct sum of indecomposable $d$-linear forms is unique up to permutation of indecomposable summands.
\end{itemize}
\end{theorem}
\noindent As a consequence, we have a simple algorithm for direct sum decompositions of arbitrary multilinear forms which is equivalent to the classical eigenvalue problem of matrices, see \cite[Algorithm 3.12]{hlyz}.

Let $T_{d,n}$ be the set of all $d$-linear forms on a $n$-dimensional linear $\k$-space. If a multilinear form is not a direct sum, then we say it is indecomposable. It is clear by Theorem \ref{cd} that $(V, \Theta) \in T_{d,n}$ is indecomposable if and only if $Z(V, \Theta)$ is a local algebra. In particular, if $(V, \Theta)$ is central, i.e. $Z(V, \Theta) \cong \k,$ then $(V, \Theta)$ is absolutely indecomposable. It was already noticed in \cite[Remark 10]{bs} that multilinear forms are more likely indecomposable. We confirm this with a help of the center theory. In fact, we show in terms of elementary algebraic geometry that almost all multilinear forms are absolutely indecomposable.

\begin{theorem} \label{cod}
The set of all central $d$-linear forms is Zariski open and dense in $T_{d,n}.$
\end{theorem}

When we consider simultaneously several $d$-multilinear forms $\Theta_i:V\times \cdots \times V\to \k$ for $i=1,\cdots,r$, it  is convenient to consider the multilinear map $(\Theta_1,\cdots,\Theta_r):V\times \cdots \times V\to \k^r$. Centers of multilinear forms can be naturally extended to multilinear maps. Motivated by \cite{bfms, bs}, we focus on direct sum decomposition and symmetric equivalence of multilinear maps. We show that the center theory is highly effective in dealing with these problems. Interestingly enough, the results of symmetric equivalence can be applied to provide a simple linear algebraic proof for a well known Torelli type result of Donagi \cite[Proposition 1.1]{d}.

\begin{theorem}\label{do}
Suppose the field $\k$ is algebraically closed and its characteristic is zero or greater than $d$. If $f$ and $g$ are two homogeneous polynomials of degree $d$ with the same Jacobian ideal, then they are related by an invertible linear transformation.
\end{theorem}

We conclude the introduction with a brief outline of the paper. In Section 2 we investigate the algebraic structure of centers of multilinear forms with applications to their direct sum decompositions. In Section 3 the theory of centers is extended to multilinear maps and is applied to their symmetric equivalence. Throughout we assume that $d>2$ is an integer, $\k$ is a field of characteristic $0$ or $>d.$ The results are presented in terms of multilinear forms. We leave the equivalent version for tensors to the interested reader.

\section{Centers and direct sum decompositions of multilinear forms}
In this section, we consider the center algebras of multilinear forms with applications to direct sum decompositions. The main aim is to prove Theorems \ref{cd} and \ref{cod}.

First of all, we recall some concepts.
\begin{definition}
Let $(V,\Theta)$ be a $d$-linear form.
If there exist nonzero subspaces $V_1,\cdots, V_s \ (s\geq 2)$  of $(V,\Theta)$  such that $V=V_1 \oplus \cdots \oplus V_s$  and  $\Theta(v_1, \cdots, v_d)=0$ unless all the $v_i$'s are in the same $V_k$ for some $k$, then $\Theta$ is  called the (inner) direct sum of $\Theta_1,\cdots, \Theta_s$, where $\Theta_i=\Theta_{|V_i}$  is the restriction of $\Theta$ to $V_i$  for $1\leq i\leq s$ and we denote it by $(V,\Theta)=(V_1,\Theta_1) \oplus \cdots \oplus (V_s,\Theta_s)$. We call $(V,\Theta)$ decomposable if it is a direct sum. Otherwise, we call $(V,\Theta)$ indecomposable.
\end{definition}

Similar to the symmetric case \cite{h1} there is no harm in assuming that the $d$-linear form $(V,\Theta)$ is nondegenerate, that is $u=0$ is the only solution to the following linear equations
\begin{equation}\label{rad}
\Theta(u,v_1, \cdots, v_{d-1})=\Theta(v_1, u, \cdots, v_{d-1})=\cdots=\Theta(v_1, \cdots, v_{d-1}, u)=0, \quad \forall v_1, \cdots, v_{d-1} \in V.
\end{equation}
For an arbitrary $d$-linear form $(V,\Theta),$ let $V_0$ be the solution space of the previous equations \eqref{rad} and take a subspace $V_1$ of $V$ such that $V=V_1 \oplus V_0,$ then $(V,\Theta)=(V_1, \Theta_1) \oplus (V_0, \Theta_0).$ In particular, $(V_1, \Theta_1)$ is nondegenerate and $(V_0, \Theta_0)$ is a zero form. Note that $V_0$ is uniquely determined by $\Theta,$ see also \cite{bs}.

Now we are ready to prove Theorem \ref{cd}.

\noindent{\bf Proof of Theorem \ref{cd}}
(1) Let us show  that $Z(V,\Theta)$ is a commutative subalgebra of $\End(V).$ It is obvious that $Z(V,\Theta)$ is closed under linear combinations. Choose two arbitrary $\phi, \psi \in Z(V,\Theta)$ and we want to  show $\phi\circ \psi \in Z(V,\Theta).$ According to the definition of centers, for all $v_1,\cdots,v_d\in V$  we have
\begin{eqnarray*}
\Theta(v_1,  \cdots, \phi\circ \psi (v_{i}), \cdots, v_{j}, \cdots v_d)
  &=& \Theta(v_1,  \cdots, \psi(v_{i}), \cdots, v_{j},\cdots, \phi( v_d))\\
  &=& \Theta(v_1, \cdots, v_i,\cdots, \psi(v_j),\cdots, \phi(v_d))  \\
  &=& \Theta(v_1, \cdots, v_i,\cdots, \phi\circ\psi(v_j),\cdots, v_d).
 \end{eqnarray*}
 Hence we have  $\phi\circ \psi \in Z(V,\Theta)$. Similarly we show the commutativity of $Z(V,\Theta).$
\begin{eqnarray*} 
\Theta(v_1,  \cdots, \phi\circ \psi (v_{i}), \cdots, v_{j}, \cdots v_d)
&=&\Theta(v_1,  \cdots, \psi(v_{i}), \cdots, \phi(v_{j}), \cdots v_d)\\
&=& \Theta(v_1, \cdots, v_i,\cdots, \phi(v_j),\cdots, \psi(v_d)) \\
&=&\Theta(v_1, \cdots, \phi(v_i),\cdots, v_j,\cdots, \psi(v_d))  \\
&=&\Theta(v_1,  \cdots, \psi\circ\phi(v_{i}), \cdots, v_{j}, \cdots v_d).
\end{eqnarray*}
We conclude that $ \Theta(v_1,  \cdots, [\phi\circ \psi-\psi\circ \phi](v_{i}), \cdots, v_{j}, \cdots v_d)=0, \forall v_1,\cdots, v_d\in V$. As $(V,\Theta)$ is nondegenerate, it follows that $\phi\circ \psi-\psi\circ \phi=0$, that is, $\phi\circ \psi=\psi\circ \phi$.

(2) Suppose there exists a direct sum decomposition $(V,\Theta)=(V_1,\Theta_1) \oplus \cdots \oplus (V_s,\Theta_s).$ For $1\leq i\leq s$, let $e_i:V \twoheadrightarrow V_i \hookrightarrow V$ be the composition of the canonical projection $V \twoheadrightarrow V_i$ and the embedding $V_i \hookrightarrow V$. Then it is obvious that $e_i^2=e_i, \ e_ie_j=0, \ \forall i \ne j,$ and by definition it is direct to verify that each $e_i \in Z(V,\Theta).$ In other words, $e_1, \dots, e_s$ are a complete set of orthogonal idempotents of $Z(V,\Theta).$

Conversely, suppose $e_1, \dots, e_s$ are a complete set of orthogonal idempotents of $Z(V,\Theta).$ Let $V_i=e_iV$ and $\Theta_i=\Theta_{|V_i}.$ Then it is not hard to verify that $(V_1,\Theta_1) \oplus \cdots \oplus (V_s,\Theta_s)$ is a direct sum decomposition of $(V,\Theta).$ Indeed, assume $v_1, \dots, v_s$ are taken from the subspaces $V_i$'s and $v_j \in V_j, \ v_k \in V_k$ with $j < k,$ then
\begin{eqnarray*}
&&\Theta(v_1, \cdots, v_j, \cdots, v_k, \cdots, v_s) \\
&=&\Theta(v_1, \cdots, e_jv_j, \cdots, e_kv_k, \cdots, v_s) \\
&=&\Theta(v_1, \cdots, v_j, \cdots, e_je_kv_k, \cdots, v_s) \\
&=&0.
\end{eqnarray*}

(3) It suffices to prove that $Z(V, \Theta)$ has a unique complete set of primitive orthogonal idempotents disregarding their order thanks to (2). Suppose $1=e_1+\cdots+e_s=f_1+\cdots+f_t$ where all $e_i$'s and $f_j$'s are primitive orthogonal idempotents. Then for any fixed $i,$ $e_i=e_i(f_1+\cdots+f_t)=e_if_1+\cdots+e_if_t.$ Since $(e_if_j)^2=e_i^2f_j^2=e_if_j$ and $e_i$ is primitive, $e_i=e_if_j$ for some certain $j.$ Similarly, $f_j=f_je_k$ for some certain $k.$ We claim that $i=k$ and so $e_i=f_j.$ Otherwise, if $i \ne k,$ then $e_i=e_if_j=e_if_je_k=e_ie_kf_j=0.$ This is absurd. Then we are done.

\begin{remarks} \label{indec} Keep the assumption that $(V, \Theta)$ is a nondegenerate $d$-linear form.
\begin{enumerate}
  \item $(V,\Theta)$ is indecomposable if and only if $Z(V,\Theta)$ is a local algebra.
  \item The uniqueness of direct sum decomposition of multilinear forms were dealt with by other approaches in \cite[Proposition 2.3]{h1} (the symmetric case) and \cite[Theorem 9]{bs}. The treatment via centers seems much more convenient.
  \item The algorithm of direct sum decomposition of symmetric multilinear forms proposed by the authors \cite[Algorithm 3.12]{hlyz2} can be extended verbatim to the present situation.
\end{enumerate}
\end{remarks}

Finally, we consider the algebraic structure of the center of a general $d$-linear form. This will provide important information for $d$-linear forms.  Our main concern is whether a general $d$-linear form is decomposable or not. It was already noticed in \cite[Remark 10]{bs} that multilinear forms are more likely indecomposable. We confirm this in terms of elementary algebraic geometry with a help of the center theory. It is shown that almost all multilinear forms have trivial center, namely the center is isomorphic to the ground field. Such multilinear forms are called central. Clearly, a central multilinear form is indecomposable by item (1) of Remarks \ref{indec}. The relevant result for symmetric multilinear forms was proved in \cite[Theorem 3.2]{hlyz2}, where the same idea can be extended to the present situation.

Explicit examples of central symmetric $d$-linear forms were constructed in \cite{hlyz2}. Here we construct some examples of central $d$-linear forms which are not symmetric. Assume $V$ is an $n$-dimensional $\k$-space with a basis $e_1,\cdots,e_n.$ Let $A=(a_{i_1 i_2 \cdots i_d})_{1 \le i_1, i_2, \dots, i_d \le n}$ be the associated tensor of $(V,\Theta)$ under the basis $e_1, e_2, \dots, e_n.$ Then  we have
\begin{eqnarray}\label{ec}
Z(V,\Theta) \cong \{X\in k^{n\times n } \mid X^T A_{i_1 \cdots \underline{i_k} \cdots \underline{i_l} \cdots i_d}=A_{i_1 \cdots \underline{i_k} \cdots \underline{i_l}\cdots i_d}X, \forall 1 \leq i_1 , \cdots,i_d \leq n\},
\end{eqnarray}
where  $A_{i_1 \cdots \underline{i_k} \cdots \underline{i_l}\cdots i_d}$ denotes the $n\times n$ matrix  $(a_{i_1 \cdots i_{k-1},i,i_{k+1},\cdots, i_{l-1},j,i_{l+1},\cdots i_d})_{1 \leq i,j \le n}.$

\begin{example}\label{s} We construct a $d$-linear form with trivial center for each $d\geq 3$ and $n\geq 2$.
	
	If $n=2$, let  $\left( a_{ij} \right)_{2\times 2}=\left(\begin{array}{cc}  1& -1\\ 1&0
	\end{array}\right)$. Let $\Theta$ be the $d$-linear form such that $a_{i_1i_2i_3\cdots i_d}=a_{i_1i_2}$ for all $1\leq i_1,\cdots, i_d\leq 2$.  An easy calculation shows that $Z(V,\Theta)\cong \k$.
	
	If $n\geq 3$, let $p_1,\cdots, p_n, q_1\cdots, q_n$ be $2n$ nonzero elements of $\k$ such that $\frac{p_j}{p_i}\neq \frac{q_j}{q_i}$ whenever $i \neq j$.	Let $A_1=(a^{(1)}_{ij}) ($resp. $A_2=(a^{(2)}_{ij}) )$ be the diagonal $n\times n$ matrix with the $a^{(1)}_{ii}=p_i$ $($resp. $a^{(2)}_{ii}=q_i) $ for $1\leq i \leq  n$. Let $A_3=(a^{(3)}_{ij})$ be the matrix with  $a^{(3)}_{ij}=1$ for all $1\leq i,j\leq n$. Let  $\Theta$ be the $d$-linear form such that $a_{1i_2i_3\cdots i_d}=a^{(1)}_{i_2i_3}, a_{2i_2i_3\cdots i_d}=a^{(2)}_{i_2i_3},a_{3i_2i_3\cdots i_d}=a^{(3)}_{i_2i_3}, a_{ii_2i_3\cdots i_d}=0$ for all $4\leq i\leq n, 1\leq i_2,i_3,\cdots, i_d\leq n$.
	Suppose $X=(x_{ij})\in Z(V,\Theta)$, then we have $X^TA_{i}=A_{i}X$ for $i=1,2,3$. Consequently, we have $p_ix_{ij}=p_jx_{ji}$ and $q_ix_{ij}=q_jx_{ji}$ for all $1\leq i,j\leq n$. As $\frac{p_j}{p_i}\neq \frac{q_j}{q_i}$ whenever $i \neq j$, we conclude that $X$ must be a diagonal matrix.  Moreover $X$ must be a mutilple of identity matrix since $X^TA_3=A_3X$.	Therefore we have $Z(V,\Theta)\cong \k$.
\end{example}

\noindent{\bf Proof of Theorem \ref{cod}}
Let $C$ be the set of all central $d$-linear forms. Clearly, $C$ is not empty by Example \ref{s}. See also \cite{hlyz2} for many other examples of central $d$-linear forms which are symmetric. The center of  $(V,\Theta)$  is  the solution space  to a system of  linear equations on $x_{ij}$'s: $X^TA_{i_1 \cdots \underline{i_k} \cdots \underline{i_l}\cdots i_d}=A_{i_1 \cdots \underline{i_k} \cdots \underline{i_l}\cdots i_d}X$  for all possible  index $i_1\cdots i_d$, where we use the same notations as in Equation \eqref{ec}. The $d$-linear form $(V,\Theta)$ is central if and only if the rank of the coefficient matrix, denoted by $B,$ of the linear system \eqref{ec} is equal to $n^2-1$. Hence $C$ is the union of all the principal open sets defined by the $(n^2-1)$-minors of $B$ regarding all $a_{i_1\cdots i_d}$'s as indeterminates. Consequently, $C$ is a nonempty Zariski open set of $T_{d,n}$ and so is dense.

\section{Centers and symmetric equivalence of $d$-linear maps}
This section is motivated by \cite{bfms, bs}. We extend the theory of centers to multilinear maps and apply it to the problems of direct sum decomposition and symmetric equivalence.

We begin by recalling some notions introduced in \cite{bfms, bs}.

\begin{definition}\label{mm}
Let $V$ and $T$ be two finite dimensional vector spaces over $\k$.
\begin{enumerate}
\item A $d$-linear map on $V$ with target $T$ is a multilinear map $\Theta:V^d\to T$ and is denoted by $(V,T, \Theta)$ or $\Theta$ for short.
\item The $d$-linear map $(V, T, \Theta)$ is called nondegenerate if $u=0$ is the only solution to the linear equations $\Theta(u,v_1, \cdots, v_{d-1})=\Theta(v_1, u, \cdots, v_{d-1})=\cdots=\Theta(v_1, \cdots, v_{d-1}, u)=0$ for all  $v_1, \cdots, v_{d-1} \in V.$
\item Suppose there exist nonzero subspaces $V_1$ and $V_2$ such that $V=V_1\oplus V_2$ and $\Theta(v_1'+v_1'',\cdots,v_d'+v_d'')=\Theta(v_1',\cdots,v_d')+\Theta(v_1'',\cdots,v_d'')$ for all $v_1',\cdots,v_d'\in V_1$ and $v_1'',\cdots,v_d''\in V_2$, then $\Theta$ is called the (inner) direct sum  of $\Theta_{|V_1}$ and $\Theta_{|V_2}$. If this is the case, then $(V,T, \Theta)$  is called decomposable. Otherwise, $(V,T, \Theta)$  is called indecomposable.
\end{enumerate}
\end{definition}

\begin{definition}\label{mm2}
Let $(U,S,\Delta)$ and $(V,T,\Theta)$ be two $d$-linear maps.
\begin{enumerate}
\item $(U,S,\Delta)$ and $(V,T,\Theta)$ are called symmetrically equivalent, denoted by $(U,S,\Delta) \simeq_s (V,T,\Theta),$ if there exist linear bijections $\phi_1,\cdots, \phi_d: U\to V$ and $\psi: S \to T$
	such that $$\psi \Delta(u_1,\cdots, u_d)=\Theta(\phi_{\sigma_1}(u_1),\cdots,\phi_{\sigma_d}(u_d))$$ for all $u_1,\cdots, u_d\in U$ and each reordering $\sigma_1,\cdots,\sigma_d$ of $1,\cdots,d$.
\item  $(U,S,\Delta)$ and $(V,T,\Theta)$ are called isomorphic, denoted by $(U,S,\Delta) \cong (V,T,\Theta),$ if there exist linear bijections $\phi: U\to V$ and $\psi:S\to T $ such that $$\psi \Delta(u_1,\cdots, u_d)= \Theta(\phi(u_1),\cdots,\phi(u_d))$$ for all $u_1,\cdots, u_d\in U$.
\item  Suppose  $S=T$. The $($outer$)$ direct sum of $(U,T,\Delta)$ and $(V,T,\Theta)$ is the multilinear map $\Delta \oplus \Theta: (U \oplus V)^d \to   T$ defined by \[ \Delta   \oplus \Theta (u_1+v_1, \dots, u_d+v_d) = \Delta(u_1, \dots, u_d) + \Theta(v_1, \dots, v_d) \] for all $u_1, \dots, u_d \in U$ and $v_1, \dots, v_d \in V.$
\end{enumerate}
\end{definition}

The notion of centers for multilinear forms can be naturally extended to multilinear maps.
\begin{definition}\label{cmm}	
For a $d$-linear map $(V, T, \Theta)$, its center is defined as
	\begin{equation}
Z(V,T,\Theta) \colon = \left\{ \phi \in \operatorname{End}(V) \left|
\begin{array}{c}
  \Theta(v_1,  \cdots, \phi(v_{i}), \cdots, v_{j}, \cdots v_d) \\
  =\Theta(v_1, \cdots, v_i,\cdots, \phi(v_j),\cdots, v_d), \\
  \forall 1\leq i,j\leq d, \ \ \forall v_1, \cdots, v_d \in V
\end{array}
\right. \right\}. \end{equation}
If $Z(V,T,\Theta) \cong \k,$ then we say the $d$-linear map $(V, T, \Theta)$ is central.
\end{definition}

It is not hard to see that Theorems \ref{cd} and \ref{cod}, as well as their proofs, can be generalized almost verbatim to multilinear maps.

\begin{theorem} \label{mcd}
Let $M_{V,d,T}$ denote the affine space of $d$-linear maps on $V$ with target $T$ and  $C_{V,d,T} \subset  M_{V,d,T}$ the subset of central $d$-linear maps.
Suppose $(V,T,\Theta) \in M_{V,d,T}$ is  nondegenerate.
\begin{itemize}
\item[(1)] The center $Z(V,T,\Theta)$ is a commutative algebra.
\item[(2)] There is a one-to-one correspondence between direct sum decompositions of $(V,T,\Theta)$ and complete sets of orthogonal idempotents of $Z(V,T,\Theta).$
\item[(3)] The decomposition of $(V,T,\Theta)$ into a direct sum of indecomposable $d$-linear maps is unique up to permutation of indecomposable direct summands.
\item[(4)] $C_{V,d,T}$ is a Zariski open and dense subset of $M_{V,d,T}.$
\end{itemize}
\end{theorem}

Now we investigate symmetric equivalence of $d$-linear maps via their centers.

\begin{proposition}\label{sc}
Suppose $(U,S,\Delta)$ and $(V,T,\Theta)$  are nondegenerate $d$-linear maps. If $(U,S,\Delta) \simeq_s (V,T,\Theta),$ then $Z(U,S,\Delta)$ is isomorphic to  $Z(V,T,\Theta).$
\end{proposition}
\begin{proof}
Let $\psi:S\to T$ and $\phi_1,\cdots, \phi_d: U\to V$ be the  linear bijections such that  $\psi \Delta(u_1,\cdots, u_d)=
  \Theta(\phi_{\sigma_1}(u_1),\cdots,\phi_{\sigma_d}(u_d))$ for all $u_1,\cdots, u_d\in U$ and each reordering $\sigma_1,\cdots,\sigma_d$ of $1,\cdots,d$. In particular,  for fixed  bijections $\phi_k$ and $\phi_l$  we have
  \begin{eqnarray*}
 &&  \Delta(u_1,\cdots, \phi_k^{-1}\phi_l(u_i),\cdots,u_j,\cdots, u_d)\\
  &=&\psi^{-1}\Theta(\phi_{\sigma_1}(u_1),\cdots,\phi_k(\phi^{-1}_k\phi_l(u_i)),\cdots,\phi_l(u_j),\cdots,\ \phi_{\sigma_d}(u_d))\ \  (\phi_{\sigma_i}=\phi_k,\phi_{\sigma_j}=\phi_l)
  \\
  &=&\psi^{-1}\Theta(\phi_{\sigma_1}(u_1),\cdots,\phi_l(u_i),\cdots,\phi_l(u_j),\cdots,\phi_{\sigma_d}(u_d))\\
  &=&\psi^{-1}\Theta(\phi_{\sigma_1}(u_1),\cdots,\phi_l(u_i),\cdots,\phi_k(\phi_k^{-1}\phi_l(u_j)),\cdots,\phi_{\sigma_d}(u_d))\\
  &=& \Delta(u_1,\cdots,u_i,\cdots,\phi^{-1}_k\phi_l(u_j),\cdots, u_d)  \ \ \ \ 
  \end{eqnarray*}
  Therefore we show that $\phi_k^{-1}\phi_l\in Z(U,S,\Delta)$ for all possible pairs $(k,l)$.
   Similarly, we can show that $\phi_k \phi_l^{-1}\in Z(V,T,\Theta)$ for all possible pairs $(k,l)$.

  For any $\phi\in Z(U,S,\Delta)$, let us show that  $\phi_k\phi \phi^{-1}_k=\phi_l\phi \phi^{-1}_l$ for all possible pairs $(k,l)$  and $\phi_k \phi\phi^{-1}_k \in Z(V,T,\Theta)$  as follows. \begin{eqnarray*}
  &&\Theta(v_1,\cdots,\phi_k\phi\phi_k^{-1}(v_i),\cdots,v_j,\cdots,v_d)\\
  &=&\psi \Delta(\phi^{-1}_{\sigma_1}(v_1),\cdots,\phi^{-1}_{l}\phi_k\phi\phi^{-1}_k(v_i),\cdots,\phi^{-1}_k(v_j),\cdots,\phi^{-1}_{\sigma_d}(v_d)) \ \ (\phi^{-1}_{\sigma_i}=\phi^{-1}_l,\phi^{-1}_{\sigma_j}=\phi^{-1}_k)\\
  &=&\psi \Delta(\phi^{-1}_{\sigma_1}(v_1),\cdots,\phi\phi^{-1}_{l}\phi_k\phi^{-1}_k(v_i),\cdots,\phi^{-1}_k(v_j),\cdots,\phi^{-1}_{\sigma_d}(v_d)) \  (\mathrm{ the \ commutativity \ of} Z(U,S,\Delta))\\
  &=&\psi \Delta(\phi^{-1}_{\sigma_1}(v_1),\cdots,\phi\phi^{-1}_{l}(v_i),\cdots,\phi^{-1}_k(v_j),\cdots,\phi^{-1}_{\sigma_d}(v_d))\\
 &= &\Theta(v_1,\cdots,\phi_l \phi \phi_l^{-1}(v_i),\cdots,v_j,\cdots,v_d).
 \end{eqnarray*}
 As $\Theta$ is nondegenerate, we have $\phi_k\phi\phi_k^{-1}=\phi_l\phi\phi^{-1}_l$ for all possible pairs $(k,l)$.

Similarly, we also have
\begin{eqnarray*}
&&\Theta(v_1,\cdots,\phi_k\phi\phi_k^{-1}(v_i),\cdots,v_j,\cdots,v_d)\\
  &=&\psi\Delta(\phi^{-1}_{\sigma_1}(v_1),\cdots,\phi^{-1}_{k}\phi_k\phi\phi^{-1}_k(v_i),\cdots,\phi^{-1}_l(v_j),\cdots,\phi^{-1}_{\sigma_d}(v_d))  \ \  (\phi^{-1}_{\sigma_i}=\phi^{-1}_k,\phi^{-1}_{\sigma_j}=\phi^{-1}_l)\\
  &=&\psi \Delta(\phi^{-1}_{\sigma_1}(v_1),\cdots,\phi\phi^{-1}_k(v_i),\cdots,\phi^{-1}_l(v_j),\cdots,\phi^{-1}_{\sigma_d}(v_d)) \\
  &=&\psi \Delta(\phi^{-1}_{\sigma_1}(v_1),\cdots,\phi^{-1}_k(v_i),\cdots,\phi\phi^{-1}_l(v_j),\cdots,\phi^{-1}_{\sigma_d}(v_d)) \\
  &=& \Theta(v_1,\cdots,v_i,\cdots,\phi_l\phi\phi_l^{-1}(v_j),\cdots,v_d)\\
  &=&\Theta(v_1,\cdots,v_i,\cdots,\phi_k\phi\phi_k^{-1}(v_j),\cdots,v_d).
\end{eqnarray*}
   Therefore we conclude  $\phi_k\phi\phi_k^{-1}\in Z(V,T,\Theta)$ for all $1\leq k\leq d$.
 Finally, we can  construct  the isomorphism $\Psi:Z(U,S,\Delta) \to Z(V,T,\Theta)$ by $  \Psi(\phi)=\phi_1 \phi \phi_1^{-1}$.
\end{proof}

Before stating the main results on symmetric equivalence of multilinear forms, we need some technical preparations particularly in commutative algebra, see for instance \cite{am,e}. A multilinear map $(V,T,\Theta)$ over $\k$ is called absolutely indecomposable if it remains indecomposable after any field extension of $\k.$ A central multilinear map is obviously absolutely indecomposable.

\begin{lemma}\label{dth}
Suppose the characteristic of $\k$ is zero or coprime to $d$. Let $A$ be a   commutative finite local $\k$-algebra with the maximal ideal $m$. Let $K=A/m$ be its residue field. Then we have $A^{\times}/(A^{\times})^d\cong K^{\times}/(K^{\times})^d$, where $A^{\times}$ (resp. $K^{\times}$) is the group of units of $A$ (resp. $K$). Moreover, if $K/\k$ is purely inseparable, then $K^{\times}/(K^{\times})^d\cong \k^{\times}/(\k^{\times})^d$.
\end{lemma}
\begin{proof}
	As $A$ is local, we have an exact sequence
$ \xymatrix@C=0.5cm{
  1 \ar[r] & 1+m \ar[rr]^{} && A^{\times} \ar[rr]^{} && K^{\times} \ar[r] & 1 }$. After tensor with $\mathbb{Z}/d\mathbb{Z}$, we obtain the exact sequence
  $$\xymatrix@C=0.5cm{
  1+m/(1+m)^d\ar[rr]^{} &&A^{\times}/(A^{\times})^d  \ar[rr]^{} && K^{\times}/(K^{\times})^d \ar[r] & 1 }.$$
	Since  the characteristic of $k$ is zero or coprime to $d$, for each $a\in 1+m$  the equation $X^d-a=0$ always has a solution in $A$ by Hensel Lemma \cite{e}. Therefore  each element of $1+m$ is a $d$-th power,  and  we have $A^{\times}/(A^{\times})^d\cong (K)^{\times}/(K^{\times})^d$.

The canonical morphism $\k \rightarrow A \rightarrow K$ induces the map $\phi: \k^{\times}/(\k^{\times})^d \to K^{\times}/(K^{\times})^d $. If $K/\k$ is  purely inseparable, we first show that $\phi$ is surjective. Let $p=\mathrm{char} \k$, then  for each $a \in K$,  there exits a $p$-th power  $p^n$ such that $b=a^{p^n}\in \k$ as $K/\k$ is purely separable.  As $(p,d)=1$, there exists two integers $\alpha,\beta$ such that $p^n\alpha+d\beta=1$. Therefore, $a=a^{p^n\alpha+d\beta}=b^{\alpha} a^{d\beta}\in \k^{\times}(K^{\times})^d$ and $\phi$ is surjective. Next if there exits an element $c\in \k$ such that $c=u^d$ for some $u\in K$. Then $u$ is a root of the separable polynomial $T^d-c$ and $\k(u)/\k$ is a separable subextension of $K/\k$. However, as  $K/\k$ is purely inseparable,  we must have $\k(u)=\k$, that is,  $u\in \k$.  Hence we show  that $\phi$ is an isomorphism and  finish the proof.
\end{proof}

\begin{proposition}\label{pi}
 A  multilinear map $(V,T,\Theta)$ is absolutely indecomposable if and only if $Z(V,T,\Theta)$ is local and its residue field is purely inseparable over $\k$.
\end{proposition}

\begin{proof}
	  A   form $\Theta$ is   absolutely indecomposable if and only if $Z(V,T,\Theta) \otimes k' $ is local for any field extension $k'/\k$ by Theorem \ref{mcd}.
	  Let $m$ be the maximal ideal of $Z(V,T,\Theta)$ and $K=Z(V,T,\Theta)/m$. As $m$ is nilpotent, $Z(V,T,\Theta) \otimes k' $ is local  if and only if $K\otimes k'$ is local. Therefore it is enough to show that $K/\k$ is purely inseparable if and only if $K\otimes k'$ is local for any field extension $k'/\k$.
	   If $K/\k$ is purely inseparable, then $p=\mathrm{char}(\k)>0$ and some $p$-th power of any element of  $K\otimes k'$  belongs to $\k\otimes k'=k'$. Hence the element of $K\otimes k'$ is either nilpotent or invertible, consequently $K\otimes k'$ is local. If $K/\k$ is not purely inseparable, then there exists a maximal  separable subextension $M/\k$ of degree $r>1$. Let $\k^{alg}$ be the algebraic closure of $\k$, and we have $K\otimes \k^{alg}\cong K\otimes_{M} (M \otimes_{\k} \k^{alg})\cong (K\otimes_M \k^{alg})^r$ which is not local.
\end{proof}

\begin{theorem}\label{sm}
Let $(U,S,\Delta)$ and $(V,T,\Theta)$ be two $d$-linear maps.
	\begin{itemize}
		\item[(1)] 	Suppose  $\Delta \simeq_s \Theta.$ Let $\Delta=\Delta_0\oplus \Delta_1 \oplus \cdots \oplus\Delta_r ($resp. $\Theta=\Theta_0\oplus \Theta_1 \oplus \cdots \oplus \Theta_s)$ be the decomposition of $\Delta ($resp. $\Theta)$  as  direct sum of zero maps and indecomposable $d$-linear maps where $\Delta_0$ and $\Theta_0$ are zero maps and the other $\Delta_i$'s and $\Theta_i$'s are indecomposable.  Then we have $r=s$, and  $\Delta_i \simeq \Theta_i$ for each $i$ after suitable  reordering  of $\Theta_i$'s.
		\item[(2)]  Suppose the characteristic of $\k$ is zero or coprime to $d.$ Let  $\Delta$ and $\Theta$ be two  absolutely indecomposable $d$-linear maps. Then $\Delta \simeq_s \Theta$ if and only if  $\Delta \cong a\Theta$  for some $a \in \k^*$.
		\item[(3)]  Suppose $\k$ is algebraically closed and  its characteristic  is zero or coprime to $d$. Then $\Delta \simeq_s \Theta$ if and only if $\Delta \cong \Theta.$
	\end{itemize}		
\end{theorem}

\begin{proof}
(1) Let $U_0=\{u\in U: \Theta(u,v_1, \cdots, v_{d-1})=\Theta(v_1, u, \cdots, v_{d-1})=\cdots=\Theta(v_1, \cdots, v_{d-1}, u)=0, \quad \forall v_1, \cdots, v_{d-1} \in V\}$. Let $\overline{U}=U/U_0$ and define the $d$-linear map $\overline{\Delta}:\overline{U}\times \cdots \times \overline{U} \to S$ by $\overline{\Delta}(\overline{u_1},\cdots,\overline{u_n})=\Delta(u_1,\cdots,u_d)$ where $u_i\in U$ is the lifting of $\overline{u_i}$ for each $1\leq i\leq n$.
 Then $\overline{\Delta}$ is nondegenerate by construction. 	Choose a subspace $U'$ of $U$ such that $U=U_0\oplus U'$ and let $\Delta'=\Delta_{|U'}$. It is easy to verify that we have an inner direct sum decomposition $(U,S,\Delta)=(U_0,S,0)\oplus (U',S,\Delta')$.  The natural morphism $U'\to U \to \overline{U}$ gives an isomorphism between $(U',S,\Delta')$ and $(\overline{U},S,\overline{\Delta})$  and we have an outer direct sum decomposition  $(U,S,\Delta)\cong (U_0,S,0)\oplus (\overline{U},S,\overline{\Delta})$. Similarly, we  construct $V_0, (\overline{V}=V/V_0,T,\overline{\Theta})$ and $(V,T,\Theta)\cong (V_0,T,0)\oplus (\overline{V},T,\overline{\Theta})$.

 Let $\psi:S\to T$ and $\phi_1,\cdots, \phi_d: U\to V$ be the  linear bijections such that  $\psi \Delta(u_1,\cdots, u_d)=
\Theta(\phi_{\sigma_1}(u_1),\cdots,\phi_{\sigma_d}(u_d))$ for all $u_1,\cdots, u_d\in U$ and each reordering $\sigma_1,\cdots,\sigma_d$ of $1,\cdots,d$.
 Then we have  $\phi_i(U_0)\subset V_0$  and $\phi^{-1}_i(V_0) \subset U_0$ for each $1 \leq i\leq d$  by their definitions. Thus $\phi_i$ induces the isomorphisms  between $U_0$ (resp. $\overline{U}$) and $V_0$ (resp. $\overline{V}$) for each $1\leq i \leq d$. Let $\overline{\phi_i}$ be induced isomorphism between $\overline{U}$ and $\overline{V}$, then $\overline{\phi_1},\cdots,\overline{\phi_d}$ gives an symmetrical equivalence between $\overline{\Delta}$ and $\overline{\Theta}$. As $U_0$ and $V_0$ are isomorphic as vector space, the zero $d$-linear maps over them are also isomorphic. Therefore we can reduce the proposition to the nondegenerate case, that is, we assume $\Delta$ and  $ \Theta$ are nondegenerate.

 By Theorem \ref{mcd}, we can  uniquely decompose  $(U,S, \Delta)=(U_1,S,\Delta_1)\oplus \cdots \oplus (U_r,S,\Delta_r)$ as  an inner direct sum of indecomposable summands such that each $\Delta_i$ is correspondent to a primitive idempotent $e_i$ of the center $Z(U,S,\Delta)$ and $U_i= e_i(U)$. As $\Delta$ and $\Theta$ are symmetrically equivalent, we have $Z(U,S,\Delta) \cong Z(V,T,\Theta)$ by sending each $\phi \in Z(U,S, \Delta)$ to $\phi_1 \phi \phi_1^{-1}$ by Proposition \ref{sc}.
 By Theorem \ref{mcd} again, we  obtain the unique decomposition  $(V,T, \Theta)=(V_1,T,\Theta_1)\oplus \cdots \oplus (V_r,T,\Theta_r)$ as  an inner direct sum of indecomposable summands such that $V_i=\phi_1 e_i \phi_1^{-1}(V)$ for each $i$. Moreover the restrictions of  $\phi_k$'s on $U_i$ give an symmetric equivalence between $\Delta_i$ and $\Theta_i$ for each $i$ since $\phi_k(U_i)=\phi_ke_i(U)=\phi_ke_i\phi_k^{-1}\phi_k(U)=\phi_1e_i\phi_1^{-1}(V)=V_i$ by the proof of Proposition \ref{sc}.

(2) If $\Delta \cong a\Theta$, it is obvious that $\Delta \simeq_s \Theta$.  Conversely, if $\Delta \simeq_s \Theta$, 	we have  $\phi_i\phi_j^{-1} \in Z(V,T,\Theta)$ for all $1\leq i,j\leq d$ which was shown in the proof of Proposition \ref{sc}. For each $i$, we can write $\phi_i=a_i\phi_1$ for some $a_i\in Z(V,T,\Theta)^{\times}$. As $(V,T,\Theta)$ is  absolutely indecomposable, its center $Z(V,T,\Theta)$ is a commutative finite  local algebra with  purely inseparable residue field by Proposition  \ref{pi}. By Lemma \ref{dth}, there exists a scale $a\in \k^{\times}$ such that the product $a^{-1}\cdot a_1\cdots a_d$  has a $d$-th root $b\in Z(V,T,\Theta)$.
	Since
	\begin{eqnarray*}
	 \psi \Delta(u_1,\cdots, u_d)&=&\Theta(\phi_{1}(u_1),\cdots,\phi_{d}(u_d))\\
	&=&\Theta(a_1\phi_1(u_1),\cdots,a_d\phi_1(u_d))\\
	&=&\Theta(a_1\cdots a_d \phi_1(u_1),\phi_1(u_2),\cdots,\phi_1(u_d))\\
	&=& a\Theta(a^{-1}\cdot a_1\cdots a_d \phi_1(u_1),\phi_1(u_2),\cdots,\phi_1(u_d))\\
	&=& a\Theta(b^d\phi_1(u_1),\phi_1(u_2),\cdots,\phi_1(u_d))\\
	&=& a\Theta(b\phi_1(u_1),b\phi_1(u_2),\cdots,b\phi_1(u_d))
\end{eqnarray*}we have $\Delta\cong a\Theta$.

 (3)The isomorphism of $d$-linear maps  clearly implies they are also symmetrically  equivalent, and we only need to show the converse.  Suppose $(U,S,\Delta)\simeq_s (V,T,\Theta)$, by item (1) we have the decompositions $(U,S,\Delta)\cong (U_0,S,0)\oplus (U_1,S,\Delta_1)\oplus \cdots \oplus (U_r, S, \Delta_r)$ and $(V,T,\Theta)\cong(V_0,T,0)\oplus (V_1,T,\Theta_1)\oplus \cdots \oplus (V_r,T,\Theta_r)$ such that all $\Delta_i$'s and $\Theta_i$'s are indecomposable and $\Delta_i \simeq_s \Theta_i$  for each $i$.  As $\k$ is algebraically closed, $\Delta_i$ and $\Theta_i$ are also absolutely  indecomposable.  Then for each $i$, we have $\Delta_i \cong a_i\Theta_i \cong \Theta_i$ for some $a_i\in \k^*$ by item (2), and consequently we prove $\Delta \cong \Theta$.
\end{proof}

As direct consequences, we recover easily the related main results of \cite{bfms,bs}.
\begin{corollary}
\begin{itemize}
	\item[(1)]
Two $d$-linear maps over the complex field $\C$ are isomorphic if and only if they are symmetrically equivalent to each other.

\item[(2)]	When $d$ is a positive odd integer, two $d$-linear maps over the real field $\R$ are isomorphic if and only if they are symmetrically equivalent to each other.
	
\item[(3)]	Assume $d$ is a positive even  integer. If two  $d$-linear maps $(U,S,\Delta)$ and $(V,T,\Theta)$ over $\R$ are symmetrically equivalent, then there exist the decompositions $\Delta\cong \Delta_1\oplus \Delta_2$ and $\Theta\cong \Theta_1\oplus \Theta_2$ such that $\Delta_1$ is isomorphic to $\Theta_1$ and $\Delta_2$ is isomorphic to $-\Theta_2$.
\end{itemize}	
\end{corollary}
\begin{proof}We only prove (3) and the others are similar. We first prove the statement in the case of indecomposable forms. Suppose  that two  $d$-linear maps $(U,S,\Delta)$ and $(V,T,\Theta)$ over $\R$ are indecomposable and  symmetrically equivalent.
		Let  $\psi:S\to T$ and  $\phi_1,\cdots, \phi_d: U\to V$ be the  linear bijections such that  $\psi \Delta(u_1,\cdots, u_d)=
	\Theta(\phi_{\sigma_1}(u_1),\cdots,\phi_{\sigma_d}(u_d))$ for all $u_1,\cdots, u_d\in U$ and each reordering $\sigma_1,\cdots,\sigma_d$ of $1,\cdots,d$.
	For each $i$, we can write $\phi_i=a_i\phi_1$ for some $a_i\in Z(V,T,\Theta)^{\times}$ by Proposition \ref{sc}. As $\Theta$ is indecomposable, the center $Z(V,T,\Theta)$
	is local and its residue field  is $\R$ or $\C$. If the residue field  is $\C$, then each $a_i$ is a $d$-th power in $Z(V,T,\Theta)$ by Lemma \ref{dth}. The same argument as Theorem \ref{sm} (2) shows that $\Delta \cong \Theta$. If the residue field is $\R$, then $a_i$ or $-a_i$ is a $d$-th power, and we similarly  have $\Delta \cong \Theta$ or $\Delta \cong -\Theta$.
	
	In general,  we have the decompositions $(U,S,\Delta)\cong (U_0,S,0)\oplus (U_1,S,\Delta'_1)\oplus \cdots \oplus (U_r,S,  \Delta'_r)$ and $(V,T,\Theta)\cong(V_0,T,0)\oplus (V_1,T,\Theta'_1)\oplus \cdots \oplus (V_r,T,\Theta'_r)$ such that all $\Delta'_i$'s and $\Theta'_i$'s are indecomposable and $\Delta_i' \simeq_s \Theta'_i$  for each $i$. By the previous argument on indecomposable forms, we have $\Delta_i'\cong \Theta_i'$ or $\Delta_i'\cong -\Theta_i'$. Let $\Delta_1$ be the direct sum of $\Delta_0$ and all $\Delta_i'$ 's such that $\Delta_i\cong \Theta_i'$, and  let $\Delta_2$ be the direct sum of all $\Delta_i'$ 's such that $\Delta_i\cong -\Theta_i'$. Similarly define $\Theta_1$ and $\Theta_2$. Combining  together, we have  $\Delta\cong \Delta_1\oplus \Delta_2$ and $\Theta\cong \Theta_1\oplus \Theta_2$ with  $\Delta_1 \cong \Theta_1$ and $\Delta_2 \cong -\Theta_2$.
\end{proof}

Finally, we apply Theorem \ref{sm} to provide a linear algebraic proof for Theorem \ref{do}, the well known Torelli type result of Donagi \cite[Proposition 1.1]{d}. Let $f(x_1, \cdots, x_n) \in \k[x_1, \cdots, x_n]$ be a homogeneous polynomial of degree $d$. Let $J(f)$ be its Jacobian ideal generated by $\frac{\partial f}{\partial x_1},\cdots,\frac{\partial f}{\partial x_n}$. The classical Torelli problem concerns about whether $J(f)$ determines $f.$

Homogeneous polynomials are naturally associated to symmetric multilinear forms and symmetric tensors, see e.g. \cite{hlyz, hlyz2}.
Write $f$ in the symmetric way
\[
f(x_1, \cdots, x_n) = \sum_{1 \le i_1, \cdots, i_d \le n} a_{i_1 \cdots i_d} x_{i_1} \cdots x_{i_d}
\]
where the $a_{i_1 \cdots i_d}$'s are symmetric with respect to their indices. Let $V=\k^n$ with a basis $e_1, \dots, e_n.$ Define the symmetric $d$-linear form $\Theta \colon V \times \cdots \times V \longrightarrow \k$ by \[\Theta(e_{i_1}, \cdots, e_{i_d})=a_{i_1 \cdots i_d}, \quad \forall 1 \le i_1, \cdots, i_d \le n.\] The pair $(V, \Theta)$ is called the associated symmetric $d$-linear form of $f$ under the basis $e_1, \cdots, e_n.$ The homogeneous polynomial $f$ and the symmetric multilinear form $(V, \Theta)$ is explicitly related as \[ f(x_1, \cdots, x_n) = \Theta\left(\sum_{1 \le i \le n}x_i e_i, \ \dots, \ \sum_{1 \le i \le n}x_i e_i\right).\]
For each $1\leq i \leq n$, let $\Theta_i$ be the $(d-1)$-linear form on $V$ such that
$$\Theta_i(v_1,\cdots,v_{d-1})=\Theta(e_i,v_1,\cdots,v_{d-1}), \forall v_1,\cdots,v_{d-1}\in V.$$  It is well known that
\begin{equation}\label{der}
\frac{1}{d}\frac{\partial f}{\partial x_i} =\Theta_i\left(\sum_{1 \le i \le n}x_i e_i, \ \cdots, \ \sum_{1 \le i \le n}x_i e_i\right).
\end{equation}
In other words, $\Theta_i$ is associated to $\frac{1}{d}\frac{\partial f}{\partial x_i}$ under the basis $e_1, \dots, e_n.$
Recall that centers can also be equivalently defined in terms of homogeneous polynomials, see \cite{h1, hlyz}. Let $H$ be  the Hessian matrix $(\frac{\partial^2 f}{\partial x_i \partial x_j})_{1 \le i,\ j \le n}$ of $f$ and define its center $Z(f)$ as  \[Z(f)=\{ X \in \k^{n \times n} \mid (HX)^T =  HX \}.\]
It is clear that $Z(V,\Theta) \cong Z(f)$.

\noindent{\bf Proof of Theorem \ref{do}}
Let   $\Theta$ and  $\Delta$ be the $d$-linear forms associated to $f$ and $g$ respectively. Let $E(f)$ be the vector space spanned by $\frac{\partial f}{\partial x_1},\cdots,\frac{\partial f}{\partial x_n}$ which is the $(d-1)$-th part of $J(f)$. If $\dim J< n$, then $f$ is degenerate and let $V_0$ be the subspace consisting  of the vectors along which the partial derivative of $f$ is zero and choose another subspace $V_1$ such that $V=V_0\oplus V_1$. Then we have a decomposition $(V,\Theta)=(V_0,0)\oplus (V_1,\Theta|_{V_1})$ by Theorem \ref{sm}. As $J(f)=J(g)$ and thanks to Equation $\eqref{der}$, we have a similar direct sum decomposition $(V,\Delta)=(V_0,0)\oplus (V_1,\Delta|_{V_1})$. Moreover the restrictions of $f$ and $g$ on $V_1$ also have the same Jacobian ideal. Therefore, we reduce the theorem to the nondegenerate case.
	
As $J(f)=J(g)$ and $\dim E(f)=n$, there exists a matrix $A=(a_{ij})_{n\times n}\in GL_n(\k)$ such that
\begin{equation} \label{jac}
(\frac{\partial g}{\partial x_1},\cdots,\frac{\partial g}{\partial x_n})=(\frac{\partial f}{\partial x_1},\cdots,\frac{\partial f}{\partial x_n})A.
\end{equation}
Define $\phi \in \End(V)$ by $\phi(e_i)=\sum\limits_{j=1}^{n}a_{ji}e_j$ for all $1\leq i \leq n.$ Note by \cite[Lemma 3.10]{hlyz2} that $A \in Z(f),$ hence $\phi \in Z(V,\Theta).$  In addition, for all $1 \le i_1, \dots, i_d \le n$ we have
\begin{eqnarray*}
\Theta(\phi(e_{i_1}), e_{i_2}, \cdots, e_{i_d}) &=& \Theta(\sum\limits_{j=1}^{n}a_{ji_{1}}e_j, e_{i_2}, \cdots, e_{i_d})  \\
&=& \sum\limits_{j=1}^{n}a_{ji_{1}} \Theta_j(e_{i_2}, \cdots, e_{i_d}) \\
&=& \Delta_{i_1} (e_{i_2}, \cdots, e_{i_d})  \quad (by \ \eqref{der} \ \& \ \eqref{jac})\\
&=& \Delta(e_{i_1}, e_{i_2}, \cdots, e_{i_d}).
\end{eqnarray*}
It follows that $\phi, \operatorname{Id}, \dots, \operatorname{Id}$ make a symmetric equivalence between $\Theta$ and $\Delta.$
Then by item (3) of Theorem \ref{sm}, $f$ and $g$ are equivalent up to an invertible linear transformation.

\end{document}